\documentclass[12pt, reqno]{amsart}
\usepackage{amssymb,verbatim,amscd,amsmath,graphicx}
\usepackage{graphicx}
\usepackage{caption}
\usepackage{subcaption}
\usepackage{pdfsync}
\usepackage{fullpage}
\usepackage{color}
\usepackage{enumitem}
\usepackage{tikz}

\usetikzlibrary{patterns}
\usetikzlibrary{calc}
\usetikzlibrary{matrix}
\usepgflibrary{arrows}
\usetikzlibrary{arrows}
\usetikzlibrary{decorations.pathreplacing}
\usetikzlibrary{decorations.pathmorphing}

\numberwithin{equation}{section}
\newtheorem{theorem}{Theorem}[section]
\newtheorem{lemma}[theorem]{Lemma}
\newtheorem{proposition}[theorem]{Proposition}
\newtheorem{corollary}[theorem]{Corollary}

\theoremstyle{definition}
\newtheorem{definition}[theorem]{Definition}

\newtheorem{def-prop}[theorem]{Definition-Proposition}
\newtheorem{remark}[theorem]{Remark}
\newtheorem{example}[theorem]{Example}

\newtheorem{notation}[theorem]{Notation}

\newtheorem*{Mysketch}{Sketch of proof} 
  {\pushQED{\qed}\begin{Mysketch}}
  {\popQED\end{Mysketch}}

\DeclareMathOperator{\reg}{reg}

\DeclareMathOperator{\Ass}{Ass}

\DeclareMathOperator{\supp}{supp}

\newcommand{\R}{\mathcal{R}}

\newcommand{\BR}{{\Big\rfloor}}
\newcommand{\BL}{{\Big\lfloor}}

\newcommand{\PP}{{\mathbb P}}
\newcommand{\ZZ}{{\mathbb Z}}
\newcommand{\NN}{{\mathbb N}}

\newcommand{\cupdot}{\mathbin{\mathaccent\cdot\cup}}

\def\alb {\mathbf {\alpha}}
\def\h {\widetilde{H}}
\def\mm{{\frak m}}

\def\pp{{\frak p}}

\def\R{{\mathcal R}}

\def\a{{\bf a}}

\def\v{{\bf v}}
\def\x{{\bf x}}

\def\1{{\bf 1}}
\def\0{{\bf 0}}

\begin{document}

\title{Symbolic powers of edge ideals of graphs}

\author{Yan Gu}
\address{Soochow University, School of Mathematical Sciences, Suzhou 215006, P.R. China.}
\email{guyan@suda.edu.cn}

\author{Huy T\`ai H\`a}
\address{Tulane University \\ Department of Mathematics \\
6823 St. Charles Ave. \\ New Orleans, LA 70118, USA}
\email{tha@tulane.edu}

\author{Jonathan L. O'Rourke}
\address{Tulane University \\ Department of Mathematics \\
6823 St. Charles Ave. \\ New Orleans, LA 70118, USA}
\email{jorourk2@tulane.edu}

\author{Joseph W. Skelton}
\address{Tulane University \\ Department of Mathematics \\
6823 St. Charles Ave. \\ New Orleans, LA 70118, USA}
\email{jskelton@tulane.edu}


\begin{abstract} Let $G$ be a graph and let $I = I(G)$ be its edge ideal. When $G$ is unicyclic, we give a decomposition of symbolic powers of $I$ in terms of its ordinary powers. This allows us to explicitly compute the Waldschmidt constant and the resurgence number of $I$. When $G$ is an odd cycle, we explicitly compute the regularity of $I^{(s)}$ for all $s \in \NN$. In doing so, we also give a natural lower bound for the regularity function $\reg I^{(s)}$, for $s \in \NN$, for an arbitrary graph $G$.
\end{abstract}

\maketitle


\section{Introduction} \label{sec.intro}

In this paper, we investigate properties and invariants of symbolic powers of edge ideals of graphs. Let $G$ be a simple graph and let $I = I(G)$ be its edge ideal. For the edge ideal $I$, the symbolic powers of $I$ are defined as follows:
$$I^{(s)} = \bigcap_{\pp \in \Ass(I)} \pp^s.$$

We shall address the following problems:
\begin{enumerate}
\item examine containments between symbolic and ordinary powers of $I$, and compute associated asymptotic invariants, such as the Waldschmidt constant and the resurgence number of $I$; and
\item determining the regularity function of symbolic powers of $I$.
\end{enumerate}
The question of containments between symbolic and ordinary powers as well as the computation of the Waldschmidt contant and the resurgence number for an ideal has been extensively studied in the literature; the list of references is too large to be exhausted, so as examples we refer the interested reader to \cite{Bocci, BH, Dum, GHV, HaHu, HH} and references therein. On the other hand, the regularity of symbolic powers of edge ideals is considerably more difficult to study than that of ordinary powers; while the regularity function for ordinary powers of edge ideals has attracted significant attention in recent years (cf. \cite{AB, ABS, Ba, BBH, Erey1, Erey2, JNS, JS, JS2, JS3, NFY} and references therein), there have been very few works addressing the regularity function for symbolic powers of edge ideals (see \cite{MT, F}).

Our work is inspired by a recent preprint of Janssen, Kamp and Vander Woude \cite{JKV}, which investigates symbolic powers of the edge ideal of an odd cycle, and a general conjecture by N.C. Minh (see, for example, \cite{CHHVbook, MV}), which states that for the edge ideal $I = I(G)$ of any graph and any $s \in \NN$,
\begin{align}
\reg I^{(s)} = \reg I^s. \label{eq.Minh}
\end{align}
Our focus is on the class of \emph{unicyclic} graphs; those are graphs with a unique cycle. If the unique cycle in a unicyclic graph $G$ is even, then $G$ is a bipartite graph, and so by \cite[Theorem 5.9]{SVV}, we know that $I(G)^{(s)} = I(G)^s$ for all $s \in \NN$. Thus, we shall pay particular attention to unicyclic graphs which contain odd cycles.

We shall show that most of main results of \cite{JKV} in fact follow from a rather nice decomposition of symbolic powers of edge ideals of unicyclic graphs. We shall also give nontrivial supportive evidence for the conjectured equality (\ref{eq.Minh}) by computing explicitly the regularity of all symbolic powers of $I(G)$ when $G$ is an odd cycle. Our first main result is stated as follows.

\noindent{\bf Theorem \ref{thm.decomposition}.}
Let $G$ be a unicyclic graph with a unique cycle $C_{2n+1} = (x_1, \dots, x_{2n+1})$, and let $I = I(G)$ be its edge ideal. Let $s \in \NN$ and write $s = k(n+1) + r$ for some $k \in \ZZ$ and $0 \le r \le n$. Then
$$I^{(s)} = \sum_{t=0}^k I^{s-t(n+1)} (x_1 \cdots x_{2n+1})^t.$$

Theorem \ref{thm.decomposition} allows us to quickly recover and extend results of \cite{JKV} on containments between symbolic and ordinary powers, on the Waldschmidt constant and the resurgence number (see Theorem \ref{thm.alpharho}).

To prove Theorem \ref{thm.decomposition}, we make use of results in \cite{FGR, HP, MRV} to show that a unicyclic graph $G$ is implosive, and thus the symbolic Rees algebra of its edge ideal is generated only in degrees 1 and $(n+1)$. It remains to determine the generator(s) of degree $(n+1)$. To this end, we give a direct description of $I(G)^{(n+1)}$ in Lemma \ref{lem.n+1} (in fact, we can also combine results of \cite{FGR, HP, MRV} for this purpose, but our description of $I(G)^{(n+1)}$ in Lemma \ref{lem.n+1} is quite elementary).

We shall also use Theorem \ref{thm.decomposition} to give the first nontrivial supportive evidence for the conjectured equality (\ref{eq.Minh}). More specifically, we prove the following result.

\noindent{\bf Theorem \ref{thm.regsymbolic}.} Let $I = I(C_{2n+1})$ be the edge ideal of an odd cycle. Then, for any $s \ge 1$, we have
$$\reg I^{(s)} = \reg I^s.$$
Particularly, it follows (see Corollary \ref{cor.regsymbolic}) that, for any $s \geq 2$,
$$\reg I^{(s)} = 2s + \BL \dfrac{2n+1}{3}\BR - 1.$$

To prove Theorem \ref{thm.regsymbolic}, we first establish a general lower bound for the regularity of symbolic powers of the edge ideal of any graph.

\noindent{\bf Theorem \ref{thm.genbound}.} Let $G$ be a simple graph with edge ideal $I = I(G)$. Let $\nu(G)$ denote the induced matching number of $G$. Then, for any $s \in \NN$, we have
$$\reg I^{(s)} \ge 2s + \nu(G) -1.$$

Theorem \ref{thm.genbound} is inspired by the general lower bound for the regularity of ordinary powers of edge ideals given in \cite{BHT}. The proof goes in the same line as that of \cite[Theorem 4.5]{BHT} with an additional subtlety in examining upper-Koszul complexes associated to symbolic powers of an edge ideal and that of the edge ideal of an induced subgraph.

The proof of Theorem \ref{thm.regsymbolic} is then completed by establishing the upper bound when $I = I(C_{2n+1})$. To this end, we make use of the decomposition in Theorem \ref{thm.decomposition}.

The paper is outlined as follows. In the next section, we collect notations and terminology. In Section \ref{sec.sym}, we prove our first main result which gives the decomposition for symbolic powers of edge ideals of unicyclic graphs. In Section \ref{sec.genbound}, we establish a general lower bound for the regularity of symbolic powers of edge ideals of graphs. The paper ends with Section \ref{sec.regcycle}, where our next main theorem is proved, giving an explicit form for the regularity of symbolic powers of the edge ideal of an odd cycle.

\noindent{\bf Acknowledgement.} The first author is supported by the National Natural Science Foundation of China (11501397), the Natural Science Foundation of Jiangsu Province (BK20140300), the Jiangsu Government Scholarship for Overseas Studies (JS-2016-209) and the Priority Academic Program Development of Jiangsu Higher Education Institutions. The second author acknowledges partial support from Simons Foundation (grant \#279786) and Louisiana Board of Regents (grant \#LEQSF(2017-19)-ENH-TR-25). This work was done when the first author was visiting the other authors at Tulane University. The authors would like to thank Tulane University for its hospitality. The authors would also like to thank Mike Janssen for pointing out a mistake in the first draft of the paper.

\section{Preliminaries} \label{sec.prel}

In this section, we collect notations and terminology used in the paper.

\noindent{\bf Graph Theory.} Throughout the paper, $G$ will denote a finite simple graph over the vertex set $V(G)$ and the edge set $E(G)$. For a vertex $x \in V(G)$, let $N_G(x) = \{y \in V(G) ~\big|~ \{x,y\} \in E(G)\}$ be its \emph{neighborhood}, and set $N_G[x] = N_G(x) \cup \{x\}$. For a subset of the vertices $W \subseteq V(G)$, $N_G(W)$ and $N_G[W]$ are defined similarly.

A subgraph $G'$ of $G$ is called an \emph{induced} subgraph of $G$ if for any vertices $x,y\in V(G')$, $\{x,y\} \in E(G') \Longleftrightarrow \{x,y\} \in E(G)$. For a collection of the vertices $W \subseteq V(G)$, we shall denote by $G[W]$ the induced subgraph of $G$ on $W$, and denote by $G-W$ the induced subgraph of $G$ on $V(G) \setminus W$.

\begin{definition} Let $G$ be a graph.
\begin{enumerate}
\item A \emph{cycle} in $G$ is a sequence of distinct vertices $x_1, \dots, x_n$ such that $\{x_i,x_{i+1}\}$ is an edge for all $i = 1, \dots, n$ (here $x_{n+1} \equiv x_1$).
\item A cycle consisting of $n$ distinct vertices is called an \emph{$n$-cycle} and often denoted by $C_n.$ We shall also use $C_n = (x_1, \dots, x_n)$ to denote the $n$-cycle whose sequence of vertices is $x_1, \dots, x_n$.
\end{enumerate}
\end{definition}

\begin{definition} Let $G$ be a graph.
\begin{enumerate}
\item A \emph{matching} in $G$ is a collection of disjoint edges. The \emph{matching number} of $G$, denoted by $\beta(G)$, is the maximum size of a matching in $G$.
\item An \emph{induced matching} in $G$ is a matching $C$ such that the induced subgraph of $G$ over the vertices in $C$ does not contain any edge other than those already in $C$. The \emph{induced matching number} of $G$, denoted by $\nu(G)$, is the maximum size of an induced matching in $G$.
\end{enumerate}
\end{definition}

\begin{definition} Let $G$ be a graph.
\begin{enumerate}
\item A collection of the vertices $W\subseteq V(G)$ is called a \emph{vertex cover} if for any edge $e \in E(G)$, $W \cap e \not= \emptyset$. A vertex cover is called \emph{minimal} if no proper subset of it is also a vertex cover.
\item The \emph{vertex cover number} of $G$, denoted by $\tau(G)$, is the smallest size of a minimal vertex cover in $G$.
\item The graph $G$ is called \emph{decomposable} if there is a proper partition of its vertices $V(G) = \cupdot_{i=1}^r V_i$ such that $\tau(G) = \sum_{i=1}^r \tau(G[V_i])$. In this case $(G[V_1], \dots, G[V_r])$ is called a \emph{decomposition} of $G$. If $G$ is not decomposable then $G$ is said to be \emph{indecomposable}.
\end{enumerate}
\end{definition}

\begin{definition} Let $G$ be a graph and let $\v = (v_1, \dots, v_{|V(G)|}) \in \NN^{|V(G)|}$.
\begin{enumerate}
\item The \emph{duplication of a vertex} $x \in V(G)$ in $G$ is the graph obtained from $G$ by adding a new vertex $x'$ and all edges $\{x',y\}$ for $y \in N_G(x)$.
\item The \emph{parallelization} of $G$ with respect to $\v$, denoted by $G^\v$, is the graph obtained from $G$ by deleting the vertex $x_i$ if $v_i = 0$, and duplicating $v_i-1$ times the vertex $x_i$ if $v_i \not= 0$.
\end{enumerate}
\end{definition}

\noindent{\bf Algebra-Combinatorics Correspondences.} Let $G$ be a graph over the vertex set $V(G) = \{x_1, \dots, x_m\}$. Let $\mathbb{K}$ be an arbitrary infinite field, and let $R = \mathbb{K}[x_1, \dots, x_m]$ be the polynomial ring associated to $V(G)$.

\begin{definition} The \emph{edge ideal} of $G$ is defined to be
$$I(G) = \langle xy ~\big|~ \{x,y\} \in E(G)\rangle \subseteq R.$$
\end{definition}

For obvious reasons, we shall often abuse notation and write $xy$ for both the edge $\{x,y\} \in E(G)$ and the monomial $xy \in R$.

A commonly-used method in commutative algebra when investigating (symbolic) powers of an ideal is to consider its (symbolic) Rees algebra.

\begin{definition} Let $I \subseteq R$ be an ideal. The \emph{Rees algebra}, denoted by $\R(I)$, and the \emph{symbolic} Rees algebra, denoted by $\R_s(I)$, of $I$ are defined to be
$$\R(I) := \bigoplus_{n \ge 0} I^nt^n \subseteq R[t] \text{ and } \R_s(I) := \bigoplus_{n \ge 0} I^{(n)}t^n \subseteq R[t].$$
\end{definition}

While the Rees algebra of an ideal is always finitely generated, this is not the case in general for the symbolic Rees algebra (see for example, \cite{R}). It is, however, known that if $I$ is a monomial ideal in a polynomial ring $R$, then the symbolic Rees algebra $\R_s(I)$ is a finitely generated algebra over $R$ (see \cite{HHT, L}). Particularly, if $I = I(G)$ is the edge ideal of a graph then the generators of $\R_s(I)$ can be described by indecomposable graphs arising from $G$. The following characterization for $\R_s(I(G))$ was given in \cite{MRV}.

\begin{theorem} \label{thm.symbolicRees}
Let $G$ be a graph over the vertex set $V(G) = \{x_1, \dots, x_m\}$. Let $I = I(G)$ be its edge ideal. Then
$$\R_s(I) = \mathbb{K}[x^\v t^b ~\big|~ G^\v \text{ is an indecomposable graph and } b = \tau(G^\v)],$$
where for $\v = (v_1, \dots, v_m)$, $x^\v t^b = x_1^{v_1} \cdots x_m^{v_m} t^b$.
\end{theorem}

A particular class of graphs of our interest consists of graphs $G$ for which the minimal generators of $\R_s(I)$ are squarefree monomials.

\begin{definition} Let $G$ be a simple graph with edge ideal $I = I(G)$. The graph $G$ is called an \emph{implosive} graph if the symbolic Rees algebra $\R_s(I)$ of $I$ is generated by monomials of the form $\x^\v t^b$, where $\v \in \{0,1\}^{|V(G)|}$.
\end{definition}

Basic implosive graphs include those that are cycles, as proved in \cite[Theorem 2.3]{FGR}, which we shall now recall.

\begin{theorem} \label{thm.implosivecycle}
If $G$ is a cycle, then $G$ is implosive.
\end{theorem}

New implosive graphs can be constructed from old ones by the following construction.

\begin{definition} Let $G_1$ and $G_2$ be graphs. Suppose that $G_1 \cap G_2 = K_r$ is the complete graph of order $r$, where $G_1 \not= K_r$ and $G_2 \not= K_r$. Then, $G_1 \cup G_2$ is called the \emph{clique-sum} of $G_1$ and $G_2$.
\end{definition}

Particularly, we have the following result from \cite[Theorem 2.5]{FGR}.

\begin{theorem} \label{thm.cliquesum}
The clique-sum of implosive graphs is again implosive.
\end{theorem}


\section{Symbolic powers of unicyclic graphs} \label{sec.sym}

In this section, we investigate symbolic powers of edge ideals of unicyclic graphs. Specifically, for a unicyclic graph $G$, we shall give a nice decomposition of the symbolic powers of $I(G)$ in terms of its ordinary powers. We shall use this decomposition to compute the Waldschmidt constant and the resurgence number of $I(G)$.

We start by the following simple observations.

\begin{lemma} \label{lem.split}
Let $I \subseteq R$ be a squarefree monomial ideal and let $y$ be a variable in $R$. Write $I = J + yH$, where $J$ and $H$ are monomial ideals and $y$ does not divide any minimal generators in $J$. Then
$$I = (I:y) \cap (J,y).$$
\end{lemma}

\begin{proof} Clearly, $I \subseteq I:y$ and $I \subseteq (J,y)$. Thus, $I \subseteq (I:y) \cap (J,y)$.

Now, consider any monomial $M \in (I:y) \cap (J,y)$. If $M$ is not divisible by $y$, then since $M \in (J,y)$, we have $M \in J \subseteq I$. On the other hand, if $M = yN$, for a monomial $N$, then since $M \in (I:y)$, i.e., $y^2N = yM \in I$ and $I$ is a squarefree monomial ideal, we must have $M = yN \in I$. Therefore, $(I:y) \cap (J,y) \subseteq I$, and the equality is proved.
\end{proof}

\begin{lemma} \label{lem.intersection}
Let $G$ be a unicyclic graph with a unique cycle $C_{2n+1}$. Let $xy$ be a leaf of $G$, where $y$ is a leaf vertex. Then for all $1 \le s \le n+1$, we have
$$(I(G-N_G[y]),x)^s \cap (I(G-y),y)^s = (I(G-y),xy)^s = I(G)^s.$$
\end{lemma}

\begin{proof} By Lemma \ref{lem.split}, we have
$(I(G-N_G[y]),x) \cap (I(G-y),y) = (I(G-y),xy) = I(G).$ This implies that for any $s \in \NN$,
$$I(G)^s \subseteq (I(G-N_G[y]),x)^s \cap (I(G-y),y)^s.$$

For simplicity of notations, let $J = I(G - N_G[y]) = I(G - \{x,y\})$ and let $K = I(G-y)$. Consider any monomial $M \in (J,x)^s \cap (K,y)^s$. It suffices to show that \begin{align}
M \in I(G)^s = (K,xy)^s = \sum_{q = 0}^s K^q (xy)^{s-q}. \label{eq.need1}
\end{align}

By the binomial expansion, there must exist $0 \le p,q \le s$ such that $M \in J^px^{s-p} \cap K^qy^{s-q}$. That is, there exist monomials $N \in J^p$ and $L \in K^q$ such that $M = Nx^{s-p} = Ly^{s-q}$. Particularly, $y^{s-q} ~\big|~ N$. Since the generators of $J = I(G - \{x,y\})$ do not involve $y$, this implies that $N' = N/y^{s-q} \in J^p$. Thus, if $p \ge q$ then $M = N' x^{s-p}y^{s-q} \in J^px^{s-p}y^{s-q} \subseteq J^p x^{s-p}y^{s-p} \subseteq K^p (xy)^{s-p}$, and (\ref{eq.need1}) follows.

Suppose now that $p < q$. Let $H = (z ~\big|~ z \in N_G(x) \setminus \{y\})$. Then, $K = J + xH$. Therefore,
\begin{align*}
N & \in K^q : (x^{s-p}) \\
& = (J+xH)^q : (x^{s-p}) \\
& = (\sum_{i=0}^q J^i(xH)^{q-i}) : (x^{s-p}) \\
& = \sum_{i=0}^q J^i H^{q-i} x^{\max\{p+q-s-i, 0\}}.
\end{align*}
It follows that there exists a $0 \le i \le q$ such that $M \in J^i H^{q-i} x^{\max\{p+q-s-i, 0\}} x^{s-p}$.

Consider the case when $q = s$. Then, for each $0 \le i \le q$, we have
$$J^iH^{q-i}x^{\max\{p+q-s-i, 0\}} x^{s-p} \subseteq J^iH^{q-i}x^{q-i}x^{s-q} \subseteq K^qx^{s-q}.$$
Thus, (\ref{eq.need1}) holds since $q = s$.

Let us now assume that $q < s \le n+1$. If $p \le i \le q$ then it can be seen that $p+q-s-i < 0$, and so
$$M = Nx^{s-p} \in J^iH^{q-i}x^{s-p} = J^iH^{q-i}x^{q-i}x^{s-q}x^{i-p} \subseteq J^iH^{q-i}x^{q-i}x^{s-q} \in K^qx^{s-q}.$$
This implies that $M \in K^qx^{s-q}$. Since the generators of $K$ do not involve $y$ and $M = Ly^{s-q}$, it follows that $L \in K^qx^{s-q}$, whence $M \in K^q (xy)^{s-q}$, and (\ref{eq.need1}) is proved.

If, on the other hand, $i < p$ then we have
$$N \in J^iH^{q-i} \cap J^p.$$
Since $i < p < q < s \le n+1$, we have $i < n-1$ (and in particular, $G$ contains no odd cycles of lengths up to $2i+3$). Thus, it follows from \cite[Lemma 4.14]{DS} that
$$N \in \sum_{j=p}^q J^jH^{q-j}.$$
Therefore,
$$M = Nx^{s-p} \in \sum_{j=p}^q J^jH^{q-j}x^{s-p} = \sum_{j=p}^q J^jH^{q-j}x^{q-j}x^{s-q}x^{j-p} \subseteq \sum_{j=p}^q J^jH^{q-j}x^{q-j}x^{s-q} \subseteq K^qx^{s-q}.$$
As before, this implies that $M \in K^q (xy)^{s-q}$, and (\ref{eq.need1}) is proved.
\end{proof}

Our next lemma extends that of \cite[Corollaries 5.3 and 5.4]{JKV} to any unicyclic graphs. This lemma, in fact, can be derived back by a careful analysis of the proof of Theorem \ref{thm.decomposition} and structures of indecomposable graphs. The proof we shall present here is, however, elementary.

\begin{lemma} \label{lem.n+1}
Let $G$ be a unicyclic graph with a unique cycle $C_{2n+1} = (x_1, \dots, x_{2n+1})$, and let $I = I(G)$ be its edge ideal. Then
\begin{enumerate}
\item For $1 \le s \le n$, we have $I^{(s)} = I^s$.
\item $I^{(n+1)} = I^{n+1} + (x_1 \cdots x_{2n+1}).$
\end{enumerate}
\end{lemma}

\begin{proof} (1) follows from \cite[Corollary 4.5]{LT} (see also \cite[Theorem 4.13]{DS}). We shall now prove (2). It is easy to see that $G$ is obtained by attaching a forest $T$ to its unique cycle $C_{2n+1}$ at zero or more vertices on the cycle. Let $k = |E(T)|$ be the number of edges in this forest. We shall use induction on $k$. For $k = 0$, the conclusion is that of \cite[Corollaries 5.3 and 5.4]{JKV}. Suppose that $k \ge 1$.

If $G$ is disconnected and $T'$ is a connected component of $G$ that does not contain the cycle $C_{2n+1}$ then the conclusion follows by using \cite[Theorem 3.4]{HNTT} and the induction hypothesis on $T \setminus T'$. Thus, we shall assume that $G$ is a connected graph.

Let $xy$ be a leaf in $T$, where $y$ is a leaf vertex and $x$ is its only neighbor (such a leaf $xy$ exists since $k \ge 1$). By Lemma \ref{lem.split}, we have
\begin{align}
I = (I:y) \cap (I(G-y), y) = (I(G-N_G[y]), x) \cap (I(G-y), y). \label{eq.minusy}
\end{align}
Since all ideals in (\ref{eq.minusy}) are squarefree monomial ideals, we have
\begin{align}
I^{(n+1)} = (I(G-N_G[y]),x)^{(n+1)} \cap (I(G-y),y)^{(n+1)}. \label{eq.sind}
\end{align}

Observe that since $y$ is a leaf in $G$, $y \not\in V(C_{2n+1})$. Thus, $G-y$ is a unicyclic graph, so by the induction hypothesis and \cite[Theorem 3.4]{HNTT}, we have
$$(I(G-y),y)^{(n+1)} = (I(G-y),y)^{n+1} + (x_1 \cdots x_{2n+1}).$$

Observe further that if $x \in V(C_{2n+1})$ then $G-N_G[y] = G-\{x,y\}$ is a forest (so, a bipartite graph) and, thus, by \cite[Theorem 5.9]{SVV} and \cite[Theorem 3.4]{HNTT} we have
$$(I(G-N_G[y]),x)^{(n+1)} = (I(G-N_G[y]),x)^{n+1}.$$
This, together with (\ref{eq.sind}), implies that
\begin{align*}
I^{(n+1)} & = (I(G-N_G[y]),x)^{n+1} \cap [(I(G-y),y)^{n+1} + (x_1 \cdots x_{2n+1})] \\
& = [(I(G-N_G[y]),x)^{n+1} \cap (I(G-y),y)^{n+1}] + [(I(G-N_G[y]),x)^{n+1} \cap (x_1 \cdots x_{2n+1})] \\
& = I^{n+1} + (x_1 \cdots x_{2n+1}).
\end{align*}
Here, the last equality follows from Lemma \ref{lem.intersection} and the fact that, since $x \in V(C_{2n+1})$, $x_1 \cdots x_{2n+1}$ can be written as a product of $x$ and $n$ edges on $C_{2n+1}$; that is, $x_1 \cdots x_{2n+1} \in (I(G-N_G[y]),x)^{n+1}$.

If, on the other hand, $x \not\in V(C_{2n+1})$ then $G-N_G[y] = G-\{x,y\}$ is itself a unicyclic graph, and so, by the induction hypothesis and \cite[Theorem 3.4]{HNTT}, we have
$$(I(G-N_G[y]),x)^{(n+1)} = (I(G-N_G[y]),x)^{n+1} + (x_1 \cdots x_{2n+1}).$$
This, together with (\ref{eq.sind}), implies that
\begin{align*}
I^{(n+1)} & = [(I(G-N_G[y]),x)^{n+1}+(x_1 \cdots x_{2n+1})] \cap [(I(G-y),y)^{n+1} + (x_1 \cdots x_{2n+1})] \\
& = [(I(G-N_G[y]),x)^{n+1} \cap (I(G-y),y)^{n+1}] + (x_1 \cdots x_{2n+1}) \\
& = I^{n+1} + (x_1 \cdots x_{2n+1}).
\end{align*}
Here, the last equality again follows from Lemma \ref{lem.intersection}. The lemma is proved.
\end{proof}

We are now ready to present the first main theorem of the paper.

\begin{theorem} \label{thm.decomposition}
Let $G$ be a unicyclic graph with a unique cycle $C_{2n+1} = (x_1, \dots, x_{2n+1})$, and let $I = I(G)$ be its edge ideal. Let $s \in \NN$ and write $s = k(n+1) + r$ for some $k \in \ZZ$ and $0 \le r \le n$. Then
$$I^{(s)} = \sum_{t=0}^k I^{s-t(n+1)} (x_1 \cdots x_{2n+1})^t.$$
\end{theorem}

\begin{proof} We first consider the case when $G$ is a connected graph. Observe that since $G$ is a unicyclic graph, $G$ can be obtained by taking the clique-sums of an odd cycle successively with $K_2$. Thus, by Theorems \ref{thm.implosivecycle} and \ref{thm.cliquesum}, we deduce that $G$ is an implosive graph. This implies that the symbolic Rees algebra $\R_s(I)$ of $I$ is generated by monomials of the form $\x^\v t^b$, where $\v \in \{0,1\}^{|V(G)|}$ and $G^\v$ (which is now necessarily an induced subgraph of $G$) is indecomposable.

It can be seen from \cite[Corollary 2a]{HP} that an induced subgraph of $G$ is indecomposable only if it is either an edge or the odd cycle of $G$. It follows from Lemma \ref{lem.n+1} that these indeed give generators for the symbolic Rees algebra $\R_s(I)$. That is, $\R_s(I)$ is generated only in degrees 1 and $(n+1)$, where the only minimal generator of degree $(n+1)$ is $(x_1 \cdots x_{2n+1}) t^{n+1}$. Particularly, for any $s \in \NN$, we have
$$I^{(s)} = \sum_{p+q(n+1) = s} I^p(I^{(n+1)})^q = \sum_{t=0}^k I^{s-t(n+1)} (x_1 \cdots x_{2n+1})^t.$$
The assertion is proved when $G$ is connected.

Suppose now that $G$ is not connected. Let $G'$ be the connected component of $G$ which contains its unique odd cycle, and let $G'' = G \setminus G'$. Then $G''$ is a forest and, particularly, $G''$ is a bipartite graph. Thus, by \cite[Theorem 5.9]{SVV}, we have $I(G'')^{(s)} = I(G'')^s$ for all $s \in \NN$. The conclusion then follows from the assertion for $G'$ and by applying \cite[Theorem 3.4]{HNTT}. The theorem is proved.
\end{proof}

\begin{example} \label{ex.decomposition}
Symbolic powers of edge ideals of graphs, not necessarily unicyclic, are in general more complicated. Their symbolic Rees algebras may also contain non-squarefree monomial generators. Let $G$ be the graph with vertex set $V(G)=\{x_1,x_2,x_3,x_4,x_5,x_6,x_7\}$ and edge set $$E(G)=\{x_1x_2,x_2x_3,x_3x_4,x_4x_5,x_5x_1,x_1x_6,x_6x_7,x_7x_1\}.$$
Then, $G$ is the clique-sum of $C_3$ and $C_5$ at $x_1$. Let $I = I(G)$. Direct computation shows that
\begin{align*}
I^{(2)} & =I^2+(x_1x_6x_7), \\
I^{(3)} & =I^3+(x_1x_2x_3x_4x_5)+I(x_1x_6x_7), \\
I^{(4)} & =I^4+I(x_1x_2x_3x_4x_5)+I^2(x_1x_6x_7)+(x_1x_6x_7)^2, \\
I^{(5)} & =I^5+I^2(x_1x_2x_3x_4x_5)+I^3(x_1x_6x_7)+I(x_1x_6x_7)^2+(x_1^2x_2x_3x_4x_5x_6x_7).
\end{align*}
\end{example}

Theorem \ref{thm.decomposition} allows us to quickly recover and extend \cite[Propositions 5.10 and 5.11]{JKV} to any unicyclic graphs.

Recall that for a homogeneous ideal $I \subseteq R$, $\alpha(I)$ denotes the least generating degree of $I$. The \emph{Waldschmidt} constant of $I$ is defined to be (where the limit is known to exist)
$$\widehat{\alpha}(I) = \lim_{s \rightarrow \infty} \dfrac{\alpha(I^{(s)})}{s}.$$
The \emph{resurgence} and \emph{asymptotic resurgence} numbers of $I$ are defined to be
$$\rho(I) = \sup \Big\{\dfrac{s}{t} ~\Big|~ I^{(s)} \not\subseteq I^t\Big\} \text{ and } \rho_a(I) = \sup \Big\{ \dfrac{s}{t} ~\Big|~ I^{(sr)} \not\subseteq I^{tr} \ \forall \ r \gg 0\Big\}.$$

\begin{theorem} \label{thm.alpharho}
Let $G$ be a unicyclic graph with a unique cycle $C_{2n+1} = (x_1, \dots, x_{2n+1})$, and let $I = I(G)$ be its edge ideal.
Then
\begin{enumerate}
\item For any $s \in \NN$,
$$\alpha(I^{(s)}) = 2s - \BL \dfrac{s}{n+1} \BR.$$
Particularly, the Waldschmidt constant of $I$ is given by
$$\widehat{\alpha}(I) = 2 - \dfrac{1}{n+1} = \dfrac{2n+1}{n+1}.$$
\item $\alpha(I^{(s)})< \alpha(I^t)$ if and only if
$I^{(s)}\nsubseteq I^t$.
\item The resurgence of $I$ is given by
$$\rho_a(I) = \rho(I) = \dfrac{2n+2}{2n+1}.$$
\end{enumerate}
\end{theorem}

\begin{proof} (1) It follows from Theorem \ref{thm.decomposition} that for any $s \in \NN$,
$$\alpha(I^{(s)}) = 2s - \BL \dfrac{s}{n+1} \BR.$$
Moreover, ${\displaystyle \frac{s}{n+1}-1\leq\BL\frac{s}{n+1}\BR\leq \frac{s}{n+1}}$. Thus,
$$\widehat{\alpha}(I)={\lim\limits_{s \to \infty}}\dfrac{\alpha(I^{(s)})}{s}= 2 - \dfrac{1}{n+1} = \dfrac{2n+1}{n+1}.$$

(2) If $\alpha(I^{(s)})< \alpha(I^t)$ then clearly $I^{(s)}\nsubseteq I^t$. Conversely, suppose that $s=k(n+1)+r$,
where $0\leq r \leq n$, and that $\alpha(I^{(s)})\geq \alpha(I^t)$. By part (1), we have $2s-k\geq 2t$. Thus, $2s-i\geq 2t$ for all $0\leq i\leq k$.

Observe that if $0\leq i\leq k$ is odd then $I^{s-i(n+1)}(x_1\cdots x_{2n+1})^i\subseteq I^{s-\frac{i+1}{2}}\subseteq I^t$. On the other hand, if $0\leq i\leq k$ is even then $I^{s-i(n+1)}(x_1\cdots x_{2n+1})^i\subseteq I^{s-\frac{i}{2}}\subseteq I^t$. Therefore, by Theorem \ref{thm.decomposition}, we have that $I^{(s)}\subseteq I^t$. The assertion is proved.

(3) Let $T=\{\frac{s}{t} ~\big|~ I^{(s)}\nsubseteq I^t\}$. For any $\frac{s}{t}\in T$, by part (2), we have $\alpha(I^{(s)})< \alpha(I^t)$. By part (1), this implies that $2s-\lfloor\frac{s}{n+1}\rfloor < 2t$. It follows that $2s-\frac{s}{n+1}< 2t$, i.e., $\frac{s}{t}< \frac{2n+2}{2n+1}$. Therefore, $\rho(I)\leq \frac{2n+2}{2n+1}$.

On the other hand, by \cite[Theorem 1.2]{GHV}, we have ${\alpha}(I)/ \widehat{\alpha}(I) \leq \rho_a(I)\leq \rho(I)$. Thus, by part (1), this gives us that $\frac{2n+2}{2n+1} \leq \rho_a(I)\leq \rho(I)$. The theorem is proved.
\end{proof}


\section{A general lower bound for symbolic powers of graphs} \label{sec.genbound}

The aim of this section is to give a general linear lower bound for the regularity of symbolic powers of edge ideals of graphs. This bound is inspired by the general lower bound for the regularity of ordinary powers of edge ideals of graphs given in \cite{BHT}. In fact, our proofs will be along the same line as that of \cite{BHT}.

We begin with an observation, whose proof we could not find elsewhere.

\begin{lemma} \label{lem.syminduced}
Let $G$ be a graph and let $H$ be an induced subgraph of $G$. Then for all $s \in \NN$, $I(H)^{(s)} \subseteq I(G)^{(s)}$.
\end{lemma}

\begin{proof} Let $I(G)=\bigcap\limits_{i=1}^rp_i$ and $I(H)=\bigcap\limits_{j=1}^tq_j$ be the primary decompositions of $I(G)$ and $I(H)$, respectively. Here, the $p_i$'s and $q_j$'s are prime ideals generated by collections of variables corresponding to minimal vertex covers of $G$ and $H$, respectively.

Observe that for each $1 \le i \le r$, $p_i\supseteq I(G) \supseteq I(H) = \bigcap\limits_{i=1}^tq_i$, and so there exists an integer $1\leq j\leq t$ such that $p_i\supseteq q_j$.
This implies that, for all $s \in \NN$,
$$I(G)^{(s)} = \bigcap\limits_{i=1}^r p_i^s \supseteq \bigcap\limits_{j=1}^t q_j^s = I(H)^{(s)}.$$
The lemma is proved.
\end{proof}

The two ingredients which we shall use in this section are upper-Koszul simplicial complexes associated to monomial ideals, and Nagata-Zariski's characterization of symbolic powers of radical ideals over a field of characteristic 0.

\begin{definition} Let $I \subseteq R = \mathbb{K}[x_1, \dots, x_m]$ be a monomial ideal and let $\alb=(\alpha_1,\ldots,\alpha_m) \in \NN^m$ be a $\NN^m$-graded degree. The \emph{upper-Koszul simplicial complex} associated to $I$ at degree $\alb$, denoted by $K^{\alb}(I)$, is the simplicial complex over $V = \{x_1, \dots, x_m\}$ whose faces are:
$$\Big\{W \subseteq V ~\Big|~ \dfrac{x^\alb}{\prod\limits_{u \in W} u} \in I\Big\}.$$
\end{definition}

Given a monomial ideal $I \subseteq R$, its $\NN^m$-graded Betti numbers are given by the following formula of Hochster (see \cite[Theorem 1.34]{MS}):
\begin{equation}\label{EQ1}
\beta_{i,\alb}(I) = \dim_\mathbb{K} \h_{i-1}(K^{\alb}(I);\mathbb{K}) \text{ for } i \ge 0 \text{ and } \alb \in \NN^m.
\end{equation}

In the next few results, for an integral vector $\a = (a_1, \dots, a_m) \in \NN^m$, set $|\a| = \sum_{i=1}^m a_i$.

\begin{lemma}[Nagata, Zariski] \label{lem.NZ}
Let $\mathbb{K}$ be a perfect field. Let $I$ be a radical ideal in a polynomial ring $R = \mathbb{K}[x_1, \dots, x_m]$. Then, for all $s \in \NN$, we have
$$I^{(s)} = \Big( f ~\Big|~ \dfrac{\partial^{|\a|}f}{\partial \x^\a} \in I \text{ for all } \a \in \NN^m \text{ with } |\a| \le s-1\Big).$$
\end{lemma}

Using (\ref{EQ1}) and Lemma \ref{lem.NZ} we obtain the following lemma, which is an analogue of \cite[Lemma 4.2]{BHT}, that was given for Betti numbers of powers of edge ideals.

\begin{lemma}\label{lem.betainduced}
Let $G$ be a graph and let $H$ be an induced subgraph of $G$. Then for any $s \ge 1$ and any $i,j \ge 0$, we have
$$\beta_{i,j}(I(H)^{(s)}) \le \beta_{i,j}(I(G)^{(s)}).$$
\end{lemma}

\begin{proof} Since field extensions are faithfully flat, we may assume that $\mathbb{K}$ is perfect. Let $m = |V(G)|$ and we shall view $I(G)$ and $I(H)$ both as ideals in the polynomial ring $R = \mathbb{K}[x_1, \dots, x_m]$.
For an $\NN^m$-graded degree $\alb=(\alpha_1,\dots,\alpha_m)$, by Lemmas \ref{lem.syminduced} and \ref{lem.NZ}, it is easy to see that $K^{\alb}(I(H)^{(s)}) \subseteq K^{\alb}(I(G)^{(s)})$.

Let $\supp(\alb) = \{x_i ~|~ \alpha_i \not= 0\}$ be the support of $\alb$. Observe that if $\supp(\alb) \subseteq V(H)$ and $W \in K^{\alb}(I(G)^{(s)})$ then
$$g = \dfrac{\x^\alb}{\prod_{x \in W} x} \in I(G)^{(s)}.$$
By Lemma \ref{lem.NZ}, this is equivalent to the condition that for all $\a \in \NN^m$ with $|\a| \le s-1$, we have
$$\dfrac{\partial^{|\a|} g}{\partial \x^{\a}} \in I(G).$$
Observe that since $\supp(\alb) \subseteq V(H)$, we have $\supp g \subseteq V(H)$. This, together with the fact that $H$ is an induced subgraph of $G$, implies that for all $\a \in \NN^m$ with $|\a| \le s-1$, we have
$$\dfrac{\partial^{|\a|} g}{\partial \x^{\a}} \in I(H).$$
That is, $W \in K^{\alb}(I(H)^{(s)})$. Thus, $K^{\alb}(I(H)^{(s)}) = K^{\alb}(I(G)^{(s)})$. Therefore, it follows from $(\ref{EQ1})$ that if $\supp(\alb) \subseteq V(H)$ then
$$\beta_{i,\alb}(I(H)^{(s)}) = \dim_\mathbb{K} \h_{i-1}(K^{\alb}(I(H)^{(s)}); \mathbb{K}) = \dim_\mathbb{K} \h_{i-1}(K^{\alb}(I(G)^{(s)}); \mathbb{K}) = \beta_{i,\alb}(I(G)^{(s)}).$$

We now have
\begin{align*} \beta_{i,j}(I(H)^{(s)}) &=\sum_{\alb\in\NN^m, \ \supp(\alb)\subseteq V(H), |\alb| = j} \beta_{i,\alb}(I(H)^{(s)})=\sum_{\alb\in\NN^m, \ \supp(\alb)\subseteq V(H), |\alb| = j} \beta_{i,\alb}(I(G)^{(s)})\\
&\le \sum_{\alb\in\NN^m, \ |\alb| = j} \beta_{i,\alb}(I(G)^{(s)}) = \beta_{i,j}(I(G)^{(s)}).
\end{align*}
\end{proof}

\begin{corollary} \label{cor.reginduced}
Let $G$ be a graph and let $H$ be an induced subgraph of $G$. Then, for all $s \ge 1$,
$$\reg I(H)^{(s)} \le \reg I(G)^{(s)}.$$
\end{corollary}

Our next main result, which gives a general linear lower bound for the regularity function of symbolic powers of edge ideals, is stated as follows. The statement of Theorem \ref{thm.genbound} is inspired by that of \cite[Theorem 4.5]{BHT}.

\begin{theorem} \label{thm.genbound}
Let $G$ be a simple graph with edge ideal $I = I(G)$. Let $\nu(G)$ denote the induced matching number of $G$. Then, for any $s \in \NN$, we have
$$\reg I^{(s)} \ge 2s + \nu(G) -1.$$
\end{theorem}

\begin{proof} Let $r=\nu(G)$. Suppose that $\{u_1v_1, \ldots, u_rv_r\}$ is an induced matching in $G$. Let $H$ be the induced subgraph of $G$ on the vertices $\bigcup_{i=1}^r\{u_i,v_i\}$. Then, $I(H)=(u_1v_1, \ldots, u_rv_r)$ is a complete intersection. Thus, for all $s \in \NN$, by \cite[Lemma 4.4]{BHT}, we have
$$\reg I(H)^{(s)} = \reg I(H)^s = 2s+r-1.$$
The conclusion now follows from Corollary \ref{cor.reginduced}.
\end{proof}

\begin{remark} It was communicated to the second named author in \cite{MV} that Minh and Vu have obtained in their unpublished work the same bound as that of Theorem \ref{thm.genbound}.
\end{remark}

\begin{example} In general, we expect that $\reg I(G)^{(s)}$ can be arbitrarily larger than the lower bound of $2s+\nu(G)-1$. Let $T$ be a graph satisfying the conditions of \cite[Theorem 3.8]{ABS}. That is, $T$ is a unicyclic graph and $\reg I(T) = \nu(T) + 2$. Fix any integer $t > 0$ and let $G$ be the disjoint union of $t$ copies of $T$. Clearly, $\nu(G) = t\nu(T)$. It also follows from \cite[Theorem 2.4]{HTT} and \cite[Theorem 1.1]{NV} that
$$\reg I(G)^s = 2s + t\reg I(T) - (t+1) = 2s + t\nu(T) + (t-1) = 2s + \nu(G) - 1 + t.$$
By the conjectured equality (\ref{eq.Minh}), we expect $\reg I(G)^{(s)}$ to be also equal to $2s + \nu(G) - 1 + t$.
\end{example}


\section{Regularity of symbolic powers of odd cycles} \label{sec.regcycle}

The aim of this section is to compute the regularity function $\reg I(G)^{(s)}$, for $s \in \NN$, when $G = C_{2n+1}$ is an odd cycle over the vertices $V(G) = \{x_1, \dots, x_{2n+1}\}$.

We shall make use of the decomposition of $I(G)^{(s)}$ given in Theorem \ref{thm.decomposition}. Particularly, let $I = I(C_{2n+1})$ be the edge ideal of $C_{2n+1}$. For any $s \in \NN$, write $s = k(n+1) + r$ for some $0 \le r \le n$. Then, by Theorem \ref{thm.decomposition}, we have $I^{(s)} = I^s + J$, where
$$J = \sum_{i=1}^k I^{s-i(n+1)} (x_1 \cdots x_{2n+1})^i.$$
Let $w = x_1 \cdots x_{2n+1}$. Then, it is easy to see that $J = wI^{(s-(n+1))}.$

Let $\mm = (x_1, \dots, x_{2n+1})$ denote the maximal homogeneous ideal in $R$. The following lemma is essential for our computation of $\reg I^{(s)}$.

\begin{lemma} \label{lem.IsJ}
Let $I = I(C_{2n+1})$ be the edge ideal of an odd cycle. For a fixed $s \in \NN$, write $s = k(n+1)+r$ for some $0 \le r \le n$ and let $J$ be as before. Then
$$I^s \cap J = w\mm I^{s-(n+1)}.$$
\end{lemma}

\begin{proof} From the description of $J$, we have
\begin{align}
I^s \cap J = \sum_{i=1}^k I^s \cap I^{s-i(n+1)}w^i = \sum_{i=1}^k w^i[I^{s-i(n+1)} \cap (I^s:w^i)]. \label{eq.capj}
\end{align}

We first claim that for any $1 \le i \le k$,
\begin{align}
I^{s-i(n+1)} \cap (I^s : w^i) = I^{s-i(n+1)}\mm^i. \label{eq.capi}
\end{align}

Indeed, it is easy to see that for any $x_j \in V(C_{2n+1})$, $x_jw = (x_jx_{j+1})(x_{j+2}x_{j+3}) \cdots (x_{j-1}x_j) \in I^{n+1}$. Thus,
$$I^{s-i(n+1)}\mm^i w^i \subseteq I^{s-i(n+1)}(I^{n+1})^i = I^s.$$
That is, $I^{s-i(n+1)}\mm^i \subseteq I^s : w^i$. This implies that
$$I^{s-i(n+1)}\mm^i \subseteq I^{s-i(n+1)} \cap (I^s : w^i).$$

Conversely, let $M \in I^{s-i(n+1)} \cap (I^s : w^i)$ be any monomial. Then, $Mw^i \in I^s$. This, in particular, implies that $\deg M \ge 2s - (2n+1)i$. On the other hand, since $M \in I^{s-i(n+1)}$, we can write $M = NL$, where $N$ is a minimal generator of $I^{s-i(n+1)}$ and, thus, is of degree $2(s-i(n+1))$. It follows that $\deg L \ge i$; that is, $L \in \mm^i$. Therefore, $M \in I^{s-i(n+1)}\mm^i$. The equality (\ref{eq.capi}) is established.

Observe further that since $x_jw \in I^{n+1}$ for any $x_j \in V(C_{2n+1})$, we have $w\mm \subseteq I^{n+1}$. Thus,
$$I^{s-(n+1)}w\mm \supseteq I^{s-2(n+1)}w^2\mm^2 \supseteq \cdots \supseteq I^{s-k(n+1)}w^k\mm^k.$$
Hence, combining with (\ref{eq.capj}) and (\ref{eq.capi}), we get $I^s \cap J = w\mm I^{s-(n+1)}.$
\end{proof}

Recall that the unique cycle $C$ in a unicyclic graph $G$ is said to be \emph{dominating} if for any $y \in V(G) \setminus V(C)$, there exists $x \in V(C)$ such that $\{x,y\} \in E(G)$; that is, $G$ is a cycle with zero or more \emph{whiskers} attached to its vertices.

\begin{remark} \label{rmk.whisker}
With the same line of arguments, Lemma \ref{lem.IsJ} can be extended to hold for any  unicyclic graph with a dominating odd cycle.
\end{remark}

Our last main result of the paper is stated as follows.

\begin{theorem} \label{thm.regsymbolic}
Let $I = I(C_{2n+1})$ be the edge ideal of an odd cycle. Then, for any $s \ge 1$, we have
$$\reg I^{(s)} = \reg I^s.$$
\end{theorem}

\begin{proof} The statement is clear if $s = 1$. Suppose that $s \ge 2$. Let $\nu = \lfloor \frac{2n+1}{3}\rfloor$ be the induced matching number of $C_{2n+1}$. By \cite[Theorem 1.2]{BHT}, we have
$$\reg I^s = 2s + \nu - 1.$$
It follows from Theorem \ref{thm.genbound} that for any $s \ge 1$,
$\reg I^{(s)} \ge 2s+\nu-1.$
Thus, it suffices to show that for any $s \ge 2$,
\begin{align}
\reg I^{(s)} \le 2s+\nu-1. \label{eq.upper}
\end{align}

Indeed, the statement is true for $s \le n$ by Lemma \ref{lem.n+1} and \cite[Theorem 1.2]{BHT}. Suppose that $s \ge n+1$. By Lemma \ref{lem.IsJ}, we have
\begin{align}
I^{(s)}/I^s \simeq J/(I^s \cap J) = wI^{(s-(n+1))}/w\mm I^{s-(n+1)} \simeq \dfrac{I^{(s-(n+1))}}{\mm I^{s-(n+1)}}(-(2n+1)). \label{eq.dropw}
\end{align}

By the induction hypothesis, we have
$$\reg I^{(s-(n+1))} \le 2s-2(n+1)+\nu,$$
where the equality can only happen if $s-(n+1) = 1$. Moreover, it follows from \cite[Theorem 2.4]{CH} that
$$\reg \mm I^{s-(n+1)} \le 1 + \reg I^{s-(n+1)} \le 2s-2(n+1)+\nu+1,$$
where the equality again can only happen if $s-(n+1) = 1$. Thus, by considering the short exact sequence
$$0 \rightarrow \mm I^{s-(n+1)} \rightarrow I^{(s-(n+1))} \rightarrow I^{(s-(n+1))}/\mm I^{s-(n+1)}\rightarrow 0,$$
we obtain
$$\reg I^{(s-(n+1))}/\mm I^{s-(n+1)} \le 2s-2(n+1)+\nu.$$
This, together with (\ref{eq.dropw}), implies that
$$\reg I^{(s)}/I^s = \reg I^{(s-(n+1))}/\mm I^{s-(n+1)} + (2n+1) \le 2s+\nu-1.$$
The conclusion now follows by considering the short exact sequence
$$0 \rightarrow I^s \rightarrow I^{(s)} \rightarrow I^{(s)}/I^s \rightarrow 0.$$
The theorem is proved.
\end{proof}

\begin{corollary} \label{cor.regsymbolic}
Let $I = I(C_{2n+1})$ be the edge ideal of an odd cycle, and let $\nu = \lfloor \frac{2n+1}{3}\rfloor$ be its induced matching number. Then, for any $s \ge 2$, we have
$$\reg I^{(s)} = 2s + \nu-1.$$
\end{corollary}

\begin{proof} The conclusion follows immediately from Theorem \ref{thm.regsymbolic} and \cite[Theorem 1.2]{BHT}.
\end{proof}

\begin{remark} By Remark \ref{rmk.whisker}, it follows from \cite[Theorem 5.4]{ABS} and the proof of Theorem \ref{thm.regsymbolic} that, for a unicyclic graph $G$ with a dominating odd cycle, we have
$$\reg I(G)^{(s)} = 2s+\nu(G)-1, \ \forall \ s \in \NN.$$
\end{remark}

We end our paper with a number of identities involving symbolic and ordinary powers of edge ideals of unicyclic graphs containing a dominating odd cycle. For such a graph $G$ with a dominating cycle $C_{2n+1} = (x_1, \dots, x_{2n+1})$, we shall explore further relationships between $I^s, I^{(s)}$ and $x_1 \cdots x_{2n+1}$. Particularly, we shall compute the colon ideal $I^s : I^{(s)}$.
To do so, we start with the following observation.

\begin{lemma} \label{lem.colonbyk}
Let $G$ be a unicyclic graph with a dominating cycle $C_{2n+1} = (x_1, \dots, x_{2n+1})$.
Let $V(G)=\{x_1,\dots,x_{2n+1},y_1,\dots,y_t\}$ and $I=I(G)$.
Then, for any $k \ge 1$, we have
$$I^{kn+k}:(x_1\cdots x_{2n+1})^k=(x_1,\dots,x_{2n+1},y_1,\dots,y_t)^k.$$
\end{lemma}

\begin{proof} Observe first that $x_i(x_1 \cdots x_{2n+1}) = (x_ix_{i+1}) (x_{i+2}x_{i+3}) \cdots (x_{i-1}x_i) \in I^{n+1}$ (with addition in the subscripts performed modulo $2n + 1$), and if $\{x_i,y_j\} \in E(G)$ then $y_j(x_1 \cdots x_{2n+1}) = (y_jx_i) (x_{i+1}x_{i+2}) \cdots (x_{i-2}x_{i-1}) \in I^{n+1}$. Thus,
$$I^{kn+k}:(x_1\cdots x_{2n+1})^k\supseteq(x_1,\ldots,x_{2n+1},y_1,\ldots,y_t)^k.$$

Observe further, by a simple counting of degrees, that the minimal generators of $I^{kn+k}:(x_1\cdots x_{2n+1})^k$ have degrees at least $k$. Therefore,
$$I^{kn+k}:(x_1\cdots x_{2n+1})^k\subseteq (x_1,\ldots,x_{2n+1},y_1,\ldots,y_t)^k.$$
Hence, we have $I^{kn+k}:(x_1\cdots x_{2n+1})^k=(x_1,\dots,x_{2n+1},y_1,\dots,y_t)^k.$
\end{proof}

\begin{example} \label{ex.5Cycle} Consider the graph $G$ with $V(G)=\{x_1,x_2,x_3,x_4,x_5,y,z\}$ and
$$E(G)=\{x_1x_2, x_2x_3,x_3x_4,x_4x_5,x_5x_1,x_1y,yz\}.$$
Then, $I^{6}:(x_1\cdots x_5)^2\neq (x _1,x_2,x_3,x_4,x_5,y,z)^2$. Thus, Lemma \ref{lem.colonbyk} does not necessarily hold for all unicyclic graphs in general.
\end{example}

In general, for unicyclic graphs with an odd cycle $C_{2n+1}$, we shall compute the colon ideal $I^{n+1} : I^{(n+1)}$.

\begin{notation}
Let $G$ be a connected unicyclic graph with a unique cycle $C_{2n+1}$. Then $G$ is obtained by attaching to $C_{2n+1}$ rooted trees $\{T_1, \dots, T_m\}$ whose roots are
at the vertices $x_{i_1},\cdots, x_{i_m} \in V(C_{2n+1})$. Define
$$\Gamma(G)=\bigcup_{j=1}^mN_{T_j}(x_{i_j}).$$
\end{notation}

\begin{proposition} Let $G$ be a connected unicyclic graph with a unique cycle $C_{2n+1}$
and let $I=I(G)$. Let $F = \bigcup_{i=1}^m T_i$. Then
$$I^{n+1}:I^{(n+1)}=(x_1,\dots,x_{2n+1})+(y ~\big|~ y\in \Gamma(G))+I(F \setminus \Gamma(G)).$$
\end{proposition}

\begin{proof} It follows from Lemma \ref{lem.n+1} that
$$I^{n+1}:I^{(n+1)}=I^{n+1}:(I^{n+1}+(x_1\cdots x_{2n+1}))=I^{n+1}:(x_1\cdots x_{2n+1}).$$
Thus, it remains to show that
$$I^{n+1}:(x_1\cdots x_{2n+1})=(x_1,\dots ,x_{2n+1})+(y ~\big|~ y\in \Gamma(G))+I(F\setminus \Gamma(G)).$$

Let $E(F)=\{f_1,\ldots, f_r\}$ be the edges of $F$. It is easy to see that
$$I^{n+1}=(I(C_{2n+1})+(f_1,\ldots,f_r))^{n+1} \subseteq I(C_{2n+1})^{n+1}+(f_1,\ldots,f_r).$$
Let $M=x_2\cdots x_{2n+1}$ and
$$J=(I(C_{2n+1})^{n+1}, f_1,\ldots,f_r): M = (I(C_{2n+1})^{n+1}: M)+(f_1: M)+\cdots +(f_r: M).$$

Observe that there are two possibilities for $(f_i:M)$. If $f_i\cap \supp(M) =\emptyset$, then $(f_i: M)=(f_i)$. On the other hand, if
$f_i\cap \supp(M) \not= \emptyset$ then $(f_i:M)$ is generated by the only vertex in $f_i \setminus \supp(M) = f_i \cap \Gamma(G) \setminus N_G[x_1]$.

Now, it can be seen that
\begin{eqnarray*}
I^{n+1}:(x_1\cdots x_{2n+1})&=&(I^{n+1}: M):x_1\\
& \subseteq & J:x_1\\
&=& (I(C_{2n+1})^{n+1}:(x_1 \cdots x_{2n+1})) + [(f_1, \dots, f_r) :M]:x_1 \\
&=&(x_1,\dots ,x_{2n+1})+(y ~\big|~ y\in \Gamma(G))+I(F\setminus \Gamma(G)).
\end{eqnarray*}
On the other hand, it is easy to see that $(x_1,\dots ,x_{2n+1})+(y ~\big|~ y\in \Gamma(G))+I(F\backslash \Gamma(G))\subseteq I^{n+1}:(x_1\cdots x_{2n+1})$. Hence, the result follows.
\end{proof}

For unicyclic graphs with a dominating odd cycle, we shall compute the colon ideal $I^s : I^{(s)}$ for all $s \ge n+1$ (obviously, for $1 \le s \le n$, $I^s = I^{(s)}$ by Lemma \ref{lem.n+1}).

\begin{proposition} \label{pro.Colon}
Let $G$ be a unicyclic graph with a dominating cycle $C_{2n+1} = (x_1, \dots, x_{2n+1})$.
Let $V(G) = \{x_1, \dots, x_{2n+1}, y_1, \dots, y_t\}$ and let $I = I(G)$.
Then, for any $s \ge n+1$ and $k = \lfloor\frac{s}{n+1}\rfloor$, we have
$$I^{s}:I^{(s)}=(x_1,\ldots,x_{2n+1},y_1,\ldots,y_t)^k.$$
\end{proposition}

\begin{proof} By Theorem \ref{thm.decomposition} and \cite[Lemma 2.12]{MMV}, we have
\begin{eqnarray*}
I^s:I^{(s)}&= &I^s:(\sum_{i=0}^k I^{s-i(n+1)}(x_1\cdots x_{2n+1})^i)\\
&=& \bigcap_{i=0}^k(I^s: I^{s-i(n+1)}(x_1\cdots x_{2n+1})^i)\\
&=& \bigcap_{i=0}^k(I^{i(n+1)}:(x_1\cdots x_{2n+1})^i)\\
&=& (x_1,\ldots,x_{2n+1},y_1,\ldots,y_s)^k.
\end{eqnarray*}
Here, the last equality follows from Lemma \ref{lem.colonbyk}.
\end{proof}

\begin{example} The conclusion of Proposition \ref{pro.Colon} does not necessarily hold for all unicyclic graphs. Let $G$ be the graph from Example \ref{ex.5Cycle}. For $s=3$ we have $k=\lfloor\frac33\rfloor=1$. Then $$I^3:I^{(3)}=(x_1,x_2,x_3,x_4,x_5,y)\neq(x_1,x_2,x_3,x_4,x_5,y,z).$$
\end{example}

\begin{proposition} \label{pro.Ism}
Let $G$ be a unicyclic graph with a dominating cycle $C_{2n+1} = (x_1, \dots, x_{2n+1})$.
Let $V(G)=\{x_1,\ldots,x_{2n+1},y_1,\ldots,y_t\}$ and let $I=I(G)$.
Then, for any $s \ge n+1$, we have
$$I^{s}=I^{(s)}\cap (x_1,\ldots,x_{2n+1},y_1,\ldots,y_t)^{2s}.$$
\end{proposition}

\begin{proof} Set $\mm = (x_1,\ldots,x_{2n+1},y_1,\ldots,y_t)$, $u=x_1\ldots x_{2n+1}$ and $k=\lfloor\frac{s}{n+1}\rfloor$. Then
\begin{eqnarray*}
I^{(s)}\cap \mm^{2s}&=& (I^s+I^{s-(n+1)}u+\cdots+I^{s-k(n+1)}u^k)\cap \mm^{2s}\\
&=& I^s \cap \mm^{2s}+I^{s-(n+1)}u \cap \mm^{2s}+\cdots +I^{s-k(n+1)}u^k\cap \mm^{2s}.
\end{eqnarray*}

For any $1\leq i\leq k$, we consider $I^{s-i(n+1)}u^i\cap \mm^{2s}$. For any generator $u^iw\in I^{s-i(n+1)}u^i \cap \mm^{2s}$, we have $w=w'w''$, where $w'$ is a product of $s-i(n+1)$ edges. It is easy to see that, since $\deg (u^iw) \ge 2s$, we have $\deg (w'') \ge i$; that is, $w''\in \mm^{i}$.
This, in particular, implies that $u^iw'' \in I^{i(n+1)}$. Hence, $u^iw\in I^{s}$. The conclusion follows.
\end{proof}

\begin{example} The conclusion of Proposition \ref{pro.Ism} does not necessarily hold for all unicyclic graphs in general. Let $G$ be the graph in Example \ref{ex.5Cycle}. Any minimal vertex cover of $G$ requires three vertices from the $5$-cycle, so for each minimal prime $\mathfrak{p}$ of $I = I(G)$, there exists a monomial in $\mathfrak{p}^3$ that divides $x_1 x_2 x_3 x_4 x_5$. Therefore $x_1 x_2 x_3 x_4 x_5 \in I^{(3)}$, and $x_1 x_2 x_3 x_4 x_5 z \in I^{(3)} \cap (x_1, \ldots, x_5,y,z)^6$. However, because the only edge containing $z$ is $yz$, we have $x_1 x_2 x_3 x_4 x_5 z \not\in I^3$, and therefore $I^3 \neq I^{(3)} \cap (x_1,\ldots,x_5, y, z)^6.$
\end{example}


\end{document}